\documentclass{fpsac} 
 
\usepackage{amsmath}
\usepackage{amssymb}
\usepackage{epsfig}
\usepackage{url}
\usepackage{color}
\usepackage{graphics}
\usepackage{verbatim}

\graphicspath{{Figures/}}

\newtheorem{thm}{Theorem}
\newtheorem{prop}[thm]{Proposition}
\newtheorem{cor}[thm]{Corollary}
\newtheorem{lemma}[thm]{Lemma}

 \newtheorem{fact}[thm]{Fact}
 
 \newcounter{exampleno}
\newenvironment{remark}{\refstepcounter{exampleno}
\smallbreak\noindent
{\sc Remark.} }

\newcommand{\commentaire}[1]{}

\def\rN{\scalebox{-1}[1]{$N$}}

\def\tO{\widetilde{O}}
\def\tP{\widetilde{P}}

  \def\cM{\mathcal{M}}
 \def\cO{\mathcal{O}}
 \def\cP{\mathcal{P}}
 
 \def\Pl{P_{\ell}}
 \def\Pr{P_{\mathrm{r}}}
 \def\cL{\mathcal{L}}
 \def\cT{\mathcal{T}}
 \def\tT{\widetilde{T}}
 \def\Omin{\cO_{\mathrm{min}}}
 \def\Pmin{\cP_{\mathrm{min}}}
 \def\Tmin{\cT_{\mathrm{min}}}
 
 \def\cN{\mathcal{N}}
 \def\cI{\mathcal{I}}
 \def\cDL{\mathcal{DL}}
 \def\cDT{\mathcal{DT}}
 
 \def\tFone{\widetilde{F}_1}

 \def\qedclaim{\hfill$\triangle$\smallskip}

\newcommand{\cacher}[1]{}

\def\Tr{X_{\mathrm{r}}}
\def\Tb{X_{\mathrm{b}}}

\begin{document}
\title[New bijective links on planar maps via orientations]{New bijective links on planar maps via orientations}

\author[\'E. Fusy]{\'Eric Fusy}
\address{\'E.F.: Dept. Math., University of British Columbia, Vancouver, BC. email: eric.fusy@inria.fr}


\begin{abstract}
This article presents new bijections on planar maps. At first a bijection is established
between bipolar orientations on planar maps and specific ``transversal structures''
on triangulations of the 4-gon with no separating 3-cycle, which are called
irreducible triangulations. This bijection specializes
 to a bijection between rooted non-separable maps and rooted irreducible
triangulations. This yields in turn a bijection between rooted loopless maps
and rooted triangulations, based on the observation that loopless maps and
 triangulations are decomposed in a similar way 
into components that are respectively 
non-separable maps and irreducible triangulations. 
This gives another bijective proof (after Wormald's construction published in 1980)
of the fact that rooted loopless maps with $n$ edges are equinumerous to rooted triangulations
with $n$ inner vertices.
\end{abstract}

\maketitle

\section{Introduction}
Planar maps (connected graphs equipped with a planar embedding) are a rich source
of structural correspondences and enumerative results. 
For counting purpose, all planar maps (shortly called maps) are here
assumed to be rooted, that is, they have a marked oriented edge with 
the outer face on its left.
As discovered by Tutte in the 60's~\cite{Tu63}, 
many families of maps (eulerian, triangulations, quadrangulations...) have strikingly simple counting sequences, which are reminiscent of formulas for families of plane trees. Tutte's method is based on generating functions. At first a recursive
decomposition of maps by root edge deletion is translated to a recurrence on the 
counting sequence, and then to a functional equation satisfied by the 
associated generating function. Solving this equation yields a formula for the 
counting sequence (by coefficient extraction).  Tutte's method has the advantage of being quite automatic, but requires an involved machinery ---the quadratic method--- to solve equations on generating functions~\cite{Tu73}.

By now there is a bijective approach, introduced by Schaeffer in his thesis~\cite{S-these}, to prove counting formulas for many  
map families. 
The main ingredients are as follows: given a family of map $\cM$, one searches a \emph{regular} combinatorial structure ---typically an orientation with simple outdegree conditions--- specific to the maps in $\cM$, which gives a way to associate
with each map of $\cM$ a tree with simple degree conditions.
 Conversely the map is recovered 
 by performing local operations on the tree so as to close the faces one 
 by one.
 The bijection ensures that the counting sequence of  $\cM$ 
  is equal to the counting sequence of the associated tree family, which has typically a closed formula involving binomial coefficients.

These enumerative techniques (Tutte's recursive method or Schaeffer's bijective approach) yield thus an extensive table of counting formulas
for families of maps. 
By examining such a table, one notices that seemingly unrelated map families
are equinumerous. This asks for bijective explanations as direct as possible to understand
the underlying structural correspondences. 
On a few instances there already exist very simple bijective constructions. 
Such bijections are of a type different from the above mentioned 
bijective constructions from trees. Indeed, 
they are from map to map and operate directly on the embedding
in a simple local way. 
  Let us mention the classical \emph{duality} construction between maps with $i$ vertices and $j$ faces
and maps with $j$ vertices and $i$ faces, the \emph{radial mapping} between maps with $n$ edges and 
4-regular maps with $n$ vertices, and the so-called \emph{trinity mapping} 
between bipartite 3-regular maps
with $2n$ vertices and eulerian maps with $n$ edges (the two latter constructions can be traced back
to Tutte~\cite{Tu63}). We have observed two further coincidences:

\begin{description}
\item[(i)] non-separable maps with $n$ edges are equinumerous to so-called irreducible triangulations  
(triangulations of the 4-gon
with no separating triangle) with $n+3$ vertices,
\item[(ii)] loopless maps with $n$ edges are
equinumerous to triangulations with $n+3$ vertices. 
\end{description}


In this article, 
we describe new bijective constructions for proving (i) and (ii). 
To wit, the bijection  for (i) can be seen as a direct map-to-map correspondence. But it borrows ideas from the above mentioned bijective constructions of maps from trees. At first we 
associate a combinatorial structure with each of  the two 
map-families: plane bipolar orientations for non-separable maps and
transversal structures for irreducible triangulations (these are defined in Section~\ref{sec:prelimiaries}). 
We show in Section~\ref{sec:bij_N_avoid} 
that these structures are closely related: to each plane bipolar orientation $O$ 
is associated in a bijective way 
a so-called $N$-avoiding plane bipolar poset $P=\phi(O)$, which yields in turn
a transversal structure $X=\phi'(P)$; the mapping $\Phi=\phi'\circ\phi$
is bijective onto so-called $N$-avoiding transversal structures.  

Then, 
the correspondence $\Phi$ specialises to a bijection, described in Section~\ref{sec:bij_2conn_irr}, 
between non-separable maps and irreducible triangulations.  The bijection, denoted
by $F_1$, is based on the two following
properties:
\begin{itemize}
\item each non-separable map has a  \emph{minimal} 
plane bipolar orientation and each irreducible triangulation has a \emph{minimal}
transversal structure; 
\item the correspondence
$\Phi$ matches the plane bipolar orientations that are minimal on their non-separable maps with the transversal structures that are minimal on their irreducible triangulations.
\end{itemize}

As described in Section~\ref{sec:bij_loopless_triang}, the bijection $F_1$ yields in turn a bijection,
denoted $F_2$, for proving (ii). 
We make
use of two classical decompositions:  
a  loopless map is decomposed into a non-separable map,
called the core, and a collection of components that are smaller loopless maps; whereas a 
 triangulation is decomposed into an irreducible triangulation, called the core, and a collection of 
 components that are smaller triangulations.
The key observation is that the two decompositions are parallel, for a convenient choice of the 
size parameters. The bijection $F_2$ is thus specified recursively: the cores are matched
by the bijection $F_1$, and the components are matched recursively by $F_2$.
Let us mention that another bijective proof of (ii) has been described by Wormald~\cite{Wo80},
still recursive but based on different principles (in~\cite{Wo80}, an isomorphism is established between the generating tree of loopless maps and the generating tree of  triangulations with an additional catalytic variable), 
see the discussion after Theorem~\ref{theo:F2}.
To sum up, the whole bijective scheme of the article is shown in Figure~\ref{fig:diagram}.

\begin{figure}
\begin{center}
\includegraphics[width=14cm]{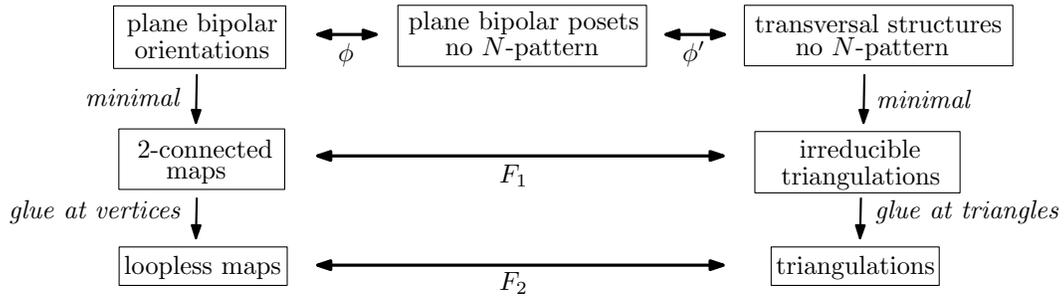}
\end{center}
\caption{Relations between the combinatorial structures and map families 
considered in the article. 
Bijective links are indicated
by double arrows. Below each double arrow is indicated the name of the 
mapping for the left-to-right arrow.}
\label{fig:diagram}
\end{figure}

\section{Preliminaries}\label{sec:prelimiaries}
\subsection{Planar maps}
A \emph{planar map}, shortly called a map, 
is a connected unlabelled planar graph embedded in the plane with no edge-crossings, the embedding 
being considered
up to continuous deformation. Loops and multiple edges are allowed.
In addition to the vertices and edges of the graph embedded, a map has 
 \emph{faces}, which are
 the maximal connected areas of the plane split by the embedding. The unbounded face is called
 the outer face, the other ones are called inner faces. Edges and vertices are said to be inner or outer
 whether they are incident to the outer face or not. A \emph{corner} 
 of a map is a triple $(e,v,e')$
 where $v$ is a vertex and $e$ and $e'$ are two consecutive edges in clockwise order
 around $v$.
 
A map is \emph{rooted} by distinguishing and orienting an edge, called the \emph{root}, with the 
condition that the root has the outer face on its right (equivalently, it is a map 
with a distinguished corner incident to the outer face). 
The origin of the root is called the \emph{root vertex}.
In this article we will consider the following families 
of (rooted) planar maps:
\begin{itemize}
\item
\noindent\emph{Loopless maps.} A loopless is a map such that each edge has two distinct extremities.
Multiple edges are allowed.
\item
\noindent\emph{non-separable maps.} A map is non-separable if it is loopless and the deletion of any one vertex
does not disconnect the map. Multiple edges are allowed.
\item
\noindent\emph{Triangulations.} A triangulation is a map with no loop nor multiple edges and with all faces
of degree 3. These correspond to maximal planar graphs embedded in the plane.
\item
\noindent\emph{Irreducible triangulations.} 
A triangulation of the 4-gon is a map with no loop nor multiple edges, with a quadrangular
outer face and triangular inner faces.
An irreducible triangulation of the 4-gon, shortly called irreducible
triangulation, is a triangulation of the 4-gon such that the interior of any 3-cycle is a face. 
\end{itemize}

\subsection{Plane bipolar orientations}\label{sec:plane_bip}
A \emph{bipolar orientation} on a connected graph $G$ is an acyclic orientation with a unique 
source (vertex with only outgoing edges) denoted by $s$ and a unique sink (vertex with only ingoing edges) denoted by $t$.
Equivalently, the partial order induced on the vertices by the orientation has a unique minimum 
and a unique maximum. Bipolar orientations constitute the natural combinatorial structure characterising 2-connectivity.
Indeed, as is well known~\cite{DeOss}, a graph $G$ with two marked vertices $s$ and $t$ admits a
bipolar orientation with source $s$ and sink $t$ if and only if $G$ is non-separable upon connecting $s$ and
$t$ by an edge.

The source and the sink of the bipolar orientation are also called the \emph{poles} or the \emph{special vertices}. The other vertices are said to be \emph{non-special}.
A \emph{plane bipolar orientation} is a bipolar orientation on a 
planar map $M$ such that the source and the sink are outer vertices of $M$.
It is convenient when considering plane bipolar orientations to draw two half-lines starting
respectively from $s$ and $t$ and reaching into the outer face. The outer face is thus split into
two unbounded faces, which are called the \emph{special faces}; looking from $s$ to $t$,  the one on the left is called the \emph{left special face} and the one on the right
is called the \emph{right special face}. 
As described in~\cite{DeOss,FeFuNoOr07}, plane bipolar orientations have the nice property that they are characterised by two simple
local properties, one around vertices and one around faces, see Figure~\ref{fig:structures}(a):

\begin{figure}
\begin{center}
\includegraphics[width=16cm]{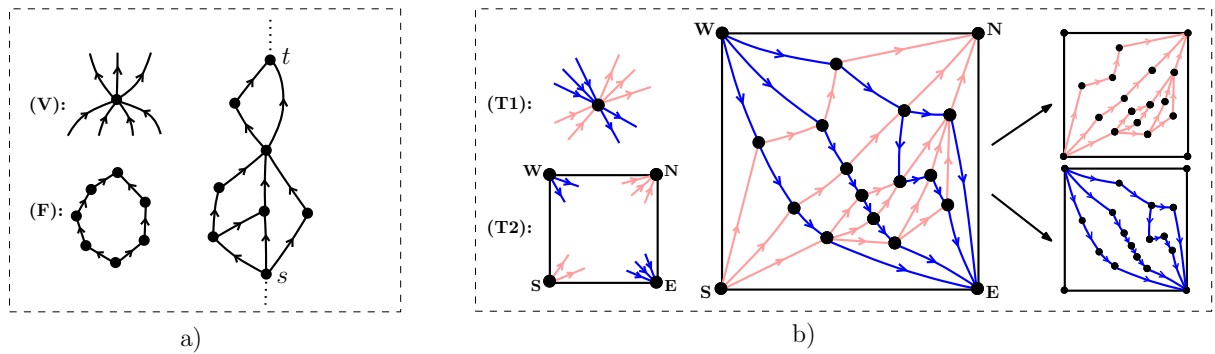}
\end{center}
\caption{(a) Plane bipolar orientations: local conditions and an example. (b) Transversal structures: local conditions and an example.}
\label{fig:structures}
\end{figure}

\begin{description}
\item[(V)]
Around each non-special vertex, the edges form a nonempty interval of outgoing edges and a
nonempty interval of ingoing edges.
\item[(F)]
The contour of each inner face $f$ is made of two oriented paths with same origin $s_f$ and
same end $t_f$. The two special faces are each 
 bordered by a path that goes 
from $s$ to $t$.  
\end{description}
Given Property~(V), we define the \emph{left lateral face} 
of a non-special vertex $v$ as the 
face in the corner 
between the last ingoing  
edge and the first outgoing  edge in clockwise order
around $v$; define similarly the \emph{right lateral face} 
of $v$ as the 
face in the corner 
between the last outgoing  
edge and the first ingoing  edge in clockwise order
around $v$.
Given an inner face $f$, 
the vertices $s_f$ and $t_f$ are respectively called the \emph{source} and the \emph{sink} of $f$.
The path from $s_f$ to $t_f$ that has the exterior of $f$ on its left 
is called the \emph{left lateral path} of $f$ and is denoted $\Pl(f)$;
the path from $s_f$ to $t_f$ that has the exterior of $f$ on its right  
is called the \emph{right lateral path} of $f$ and is denoted $\Pr(f)$. 
The last edge
of $\Pl(f)$ is called the \emph{topleft edge} of $f$ and the first edge of $\Pr(f)$ is called 
the \emph{bottomright edge} of $f$. The vertices of $\Pl(f)\backslash\{s_f,t_f\}$
 are called \emph{left lateral vertices} of $f$, and the vertices of 
 $\Pr(f)\backslash\{s_f,t_f\}$
 are called \emph{right lateral vertices} of $f$.
The paths bordering the left special face and the right special face
are called respectively  the \emph{left outer path} and the \emph{right outer path}.

It is well known that a plane bipolar orientation $O$ induces a partial order on the edge-set:
$e\leq e'$ if and only if there exists an oriented path passing by $e$ before passing by $e'$. 
In other words $\leq$ is the transitive closure of the binary relation: $e\prec e'$ if and only if there is a vertex $v$
such that $e$ is ingoing at $v$ and $e'$ is outgoing at $v$.

Due to a classical duality relation satisfied by plane 
bipolar orientations~\cite{DeOss}, there is a partial order on the face-set and 
another partial order, called dual, on the edge-set.
Precisely, given two faces $f$ and $f'$ (special or not), write $f\prec_F f'$ if
there is an edge of $O$ with $f$ on its left and $f'$ on its right; the transitive closure
of $\prec_F$ is a partial order, called the
\emph{left-to-right order} on the faces of $O$. 
And the dual order on the edge-set, denoted $\leq^*$, 
is defined as the transitive closure of the
 relation: $e\prec^* e'$ (with $e$ and $e'$ edges of $O$) 
 if and only if there exists a face $f$ of $O$ such that $e$ is in the left lateral path
and $e'$ is in the right lateral path of $f$.

Given a plane bipolar orientation with no multiple edges, 
a transitive edge is an edge whose two extremities are 
connected by an oriented path of length at least 2. 
A \emph{plane bipolar poset} is a plane bipolar orientation with at least 3 vertices, with no multiple edges, and 
with no transitive edges. 
The terminology refers to the fact that  
a plane bipolar poset is a planarly embedded Hasse diagram representing a poset.

The following property is easily checked from Condition (F):
\begin{fact}\label{fact:poset}
A plane bipolar
orientation with at least 3 vertices   
is a plane bipolar poset if and only if the two lateral paths of each inner face have length at least 2.
 \end{fact} 
 
An important remark to be used later is that
 each inner face $f$ of a plane bipolar poset has at least one lateral vertex on each side: the
 last vertex of $\Pl(f)\backslash\{s_f,t_f\}$ is called the \emph{topleft lateral vertex} of $f$ and the
 first vertex of $\Pr(f)\backslash\{s_f,t_f\}$ is called the \emph{bottomright lateral vertex} of $f$.
 Given a plane bipolar poset, a \emph{$N$-pattern} is a triple $(e_1,e_2,e_3)$ of edges
 such that $e_1$ and $e_2$ have the same origin $v$, $e_2$ and $e_3$
 have the same end $v'$, $e_1$ follows $e_2$
 in clockwise order around $v$, and $e_3$ follows $e_2$ in clockwise order around $v'$. The edge $e_2$ is called the \emph{central edge} of the 
 $N$-pattern. 
 A \emph{\rN-pattern} is defined similarly, upon replacing clockwise by counterclockwise.
Plane bipolar posets with no $N$-pattern are said to be \emph{$N$-avoiding}; these play an important role in the bijections to be given next.

\subsection{Transversal structures}\label{sec:trans_struct}

Transversal structures  
 play a similar part for irreducible triangulations as plane bipolar orientations
for non-separable maps.
%

Let us give the precise definition. 
Given a rooted irreducible triangulation $T$, denote by $N$, $E$, $S$, $W$ the outer vertices of $T$ in 
clockwise order around the outer face, starting from the origin of the root. 
A \emph{transversal structure} of $T$ is an orientation and a bicoloration of the inner edges of $T$, say each inner edge is red or blue, such that the following conditions are satisfied, see Figure~\ref{fig:structures}(b)~\footnote{In all figures, red edges are
light and blue edges are dark.}.
\begin{description}
\item[(T1)]
The edges incident to an inner vertex of $T$ form in clockwise order: a nonempty interval of
outgoing red edges, a nonempty interval of outgoing blue edges, a nonempty interval of 
ingoing red edges, and a nonempty interval of ingoing blue edges. 
\item[(T2)]
The edges incident to $N$, $E$, $S$, and $W$ are respectively ingoing red, ingoing blue,
outgoing red, and outgoing blue.
\end{description} 
 Transversal structures were introduced by He~\cite{He93} under the name of regular 
 edge-labellings.
They were further investigated by the author~\cite{Fu07b} and have many applications in graph drawing:
straight-line drawing~\cite{Fu07b}, visibility drawing~\cite{Kant}, rectangular layouts~\cite{Kant}.
Transversal structures characterise triangulations of the 4-gon that are irreducible. 
Indeed, irreducibility is necessary, 
and each irreducible triangulation admits a transversal structure~\cite{Fu07b}.
Moreover, transversal structures are closely related to plane bipolar orientations by the following property:

\begin{fact}\label{fact:trans_poset}
Let $X$ be a transversal structure on an irreducible triangulation $T$ having at least one 
inner vertex. Then the oriented map $\Tr$ 
formed by the red
edges and the vertices of $T\backslash\{W,E\}$ is a plane bipolar poset, called the \emph{red
bipolar poset} of $X$. 
Similarly,   the oriented map $\Tb$ 
formed by the blue
edges and the vertices of $T\backslash\{S,N\}$ is a plane bipolar poset, called the \emph{blue 
bipolar poset} of $X$.
\end{fact}  
\begin{proof}
The fact that $\Tr$ and $\Tb$ are plane bipolar orientations was proved in~\cite{Fu07b}
(the main argument is that a monocoloured circuit of $X$ would have a monocolored chordal path of the 
other color inside,
which is not possible by Condition~T1). Clearly in each inner face $f$  of $\Tr$, the two lateral paths
are not both of length 1, as $T$ has no multiple edge. 
Hence, $f$ has at least one lateral vertex $v$ on one
side (left or right), so $f$ has at least one blue edge inside by Condition (T1).
But, by condition (T1) again, such a blue edge goes from a left lateral vertex of $f$
to a right lateral vertex of $f$. Hence both lateral paths of $f$ have length greater
than $1$, so $\Tr$ is a plane bipolar poset by Fact~\ref{fact:poset}.
\end{proof}
A transversal structure is called \emph{$N$-avoiding} if both its red and its blue bipolar posets
are $N$-avoiding.

\section{Bijections between plane bipolar orientations and $N$-avoiding structures}\label{sec:bij_N_avoid}
This section covers the top-line of the diagram shown in Figure~\ref{fig:diagram}.
We present a bijective mapping $\Phi$ from plane bipolar orientations
to $N$-avoiding transversal structures; $\Phi$ is the composition of two 
bijections $\phi$ and $\phi'$: $\phi$ maps a plane bipolar orientation
to an $N$-avoiding plane bipolar poset, and $\phi'$ completes an $N$-avoiding
plane bipolar poset into an $N$-avoiding transversal structure.

\subsection{From plane bipolar orientations to $N$-avoiding plane bipolar posets}
Let $O$ be a plane bipolar orientation. We associate with $O$ an oriented
 planar map $P=\phi(O)$ as follows, see Figure~\ref{fig:phi1}(a)-(c).
 
 \begin{figure}
\begin{center}
\includegraphics[width=16cm]{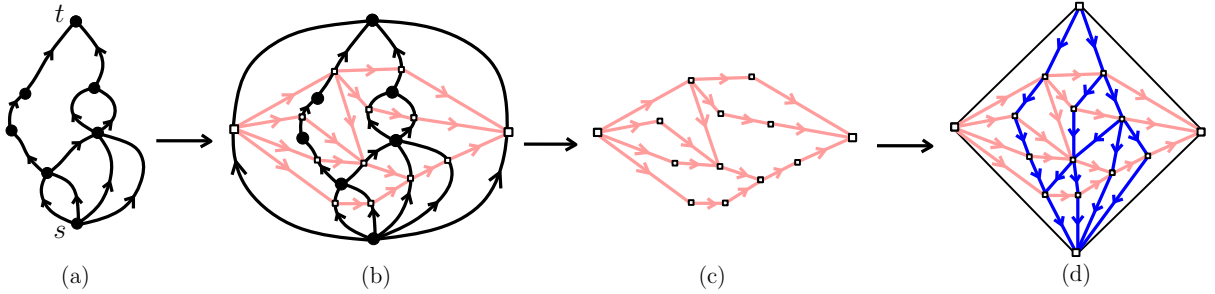}
\end{center}
\caption{(a) A plane bipolar orientation, (c) the associated $N$-avoiding plane bipolar poset, (d) the
associated $N$-avoiding transversal structure.}
\label{fig:phi1}
\end{figure}

\begin{itemize}
\item
\emph{Add two outer edges.} 
Add two edges $\ell$ and $r$ 
going from the source to the sink of $O$, so that the whole map is contained
in the 2-cycle delimited on the left by $\ell$ and on the right  by $r$. 
The augmented bipolar orientation is 
denoted $\widetilde{O}$.
\item
\emph{Insert the vertices of $P$.} A vertex of $P$, depicted in white, is inserted in the middle
of each edge of $\widetilde{O}$. The vertex inserted in the edge $\ell$  is denoted~$s_P$, and the vertex  inserted in the edge $r$ 
is denoted~$t_P$. 
\item
\emph{Insert the edges of $P$.} Edges of $P$ are planarly 
inserted in the interior of each inner face $f$ of $\tO$
 so as to create the following adjacencies, see Figure~\ref{fig:local}(a):
the vertices of $P$ in the left lateral path of $f$
are connected to the vertex of $P$ in the bottomright edge of $f$;
and the vertices of $P$ in the right lateral path of $f$
are connected to the vertex of $P$ in the topleft edge of $f$.
The inserted edges are directed from the left lateral path to the right lateral path of $f$.
\end{itemize}

\begin{figure}
\begin{center}
\includegraphics[width=16cm]{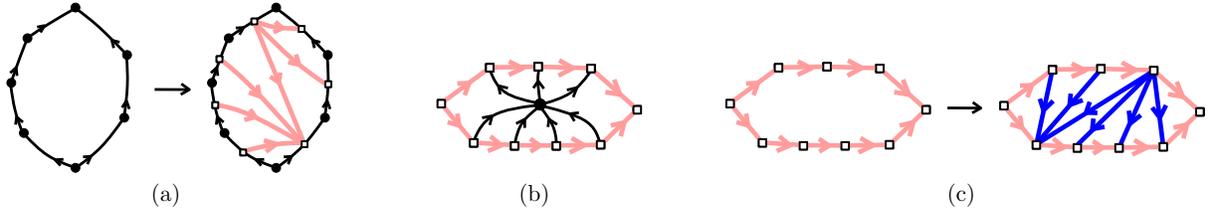}
\end{center}
\caption{(a) Insertion of the edges of $P=\phi(O)$ inside
a face of $O$. (b) Configuration of an inner face of $P$ around the associated non-special vertex of $O$.
(c) Insertion of the blue edges of $X=\phi'(P)$ inside a face of $P$.}
\label{fig:local}
\end{figure}

\begin{lemma}\label{lem:prop_phi1}
Let $O$ be a plane bipolar orientation with $n$ edges.
Then the oriented planar map $P=\phi(O)$ is an $N$-avoiding plane 
bipolar poset with $n+2$ vertices. 
The source of $P$ 
is the vertex $s_P$  and the sink of $P$ is the vertex $t_P$. 

When $P$ and the augmented bipolar orientation 
$\widetilde{O}$ are superimposed, there is 
exactly one non-special vertex $v$ of $O$ in each inner face $f$ of $P$: the outgoing edges of $v$ are incident to the left lateral vertices of $f$ and the ingoing edges of $v$ are incident to the right lateral vertices of $f$,
see Figure~\ref{fig:local}(b).
\end{lemma}
\begin{proof}
Let us first prove that $P$ is acyclic. Given two vertices $v$ and $v'$ of $P$, 
write $v\prec_P v'$ if there is an edge of $P$ from $v$ to $v'$. 
Let $e$ and $e'$ be the edges of $O$ that respectively correspond to $v$ and $v'$.
Note that $v\prec_P v'$ implies  $e\prec^* e'$. Hence the transitive closure
of $\prec_P$ corresponds to a suborder of $\leq^*$, so $P$ is acyclic.
To prove bipolarity, observe that, for each vertex $v$ of $P$ distinct from $s_P$ and $t_P$,
the corresponding edge $e$ of $O$ has 
one inner face of $\tO$ on each side. Hence, by definition of $\phi$, $v$ has at least one
ingoing edge in the face on the left of $e$ (possibly more than one if $e$ is the bottomright edge of that face) and at least one outgoing edge in the face of $\tO$ on the
right of $e$ (possibly more than one if $e$ is the topleft edge of that face).

Next, the rules for inserting the edges of $P$ 
easily imply that each non-special vertex $v$ of $O$ gives rise to an inner face $f_v$ of $P$ such that the edges of $P$ outgoing at $v$ are incident 
 to the left lateral vertices of $f_v$ and the edges of $P$ ingoing at $v$ are incident
 to the right lateral vertices of $f_v$, see Figure~\ref{fig:local}(b). 
 In addition the left lateral path of $f_v$
 has length $\mathrm{Outdeg}(v)+1$  and the right lateral
path of $f_v$ has length $\mathrm{Indeg}(v)+1$. Hence every inner face of $P$
of the form $f_v$ has its two lateral paths of length greater than $1$.
  \smallskip
 
\noindent{\bf Claim.}
The number of non-special vertices of $O$ is equal to 
 the number of inner faces of $P$.
 \smallskip
 
\noindent{\it Proof of the claim.} 
Write $V(M)$, $E(M)$ and $F(M)$ for the sets of
vertices, edges, and inner faces of a map $M$ and write $\mathrm{deg}(f)$ for the 
 number of edges on the contour of an inner face $f$. 
By definition of $\phi$, $|V(P)|=|E(\tO)|$ and 
$E(P)=\sum_{f\in F(\widetilde{O})}(\mathrm{deg}(f)-1)=2|E(\widetilde{O})|-2-|F(\tO)|$.
By the Euler relation, $|F(P)|=|E(P)|-|V(P)|+1$, so $|F(P)|=|E(\widetilde{O})|-|F(\widetilde{O})|-1$, which is equal to $(|V(\tO)|-2)$ again by the Euler relation.
Hence, the number of inner faces of $P$ is equal to the number
of non-special vertices of $O$.
\qedclaim

The claim ensures that every inner face of $P$ is of the form $f_v$ for some
non-special vertex $v$ of $O$. Hence every inner face of $P$ has its two lateral
paths of lengths greater than $1$, so  $P$ is a plane bipolar poset by Fact~\ref{fact:poset}.

Finally, the fact that $P$ has no $N$-pattern is due to the two following observations:
i) two edges of $P$ with same origin are inside the same face of $O$, and two edges of $P$
with same end are inside the same face of $O$, hence any $N$-pattern must be 
inside a face of $O$;
ii)  there is no $N$-pattern of $P$ inside a face of $O$. 
 \end{proof}

As we prove next, the mapping $\phi$ is in fact a bijection 
and has an explicit simple inverse.

Given a bipolar poset $P$ ($N$-avoiding or not), let $O=\psi(P)$ 
be the oriented planar map
defined as follows:
\begin{itemize}
\item
\emph{Insert the vertices of $O$.} One vertex of $O$ is inserted in each face of $P$.
The vertex inserted in the right special face of $P$ is denoted $s$ and the vertex inserted in the left special
face of $P$ is denoted $t$.
\item
\emph{Insert the edges of $O$.} Each non-special vertex $v$ of $P$ gives rise to an edge of $O$ that goes from the vertex of $O$ inside the right lateral face of $v$ to the vertex of $O$  inside the left lateral face of $v$. 
\end{itemize}

The correspondence between  non-special vertices of $O$ and inner faces of $P$, as stated in Lemma~\ref{lem:prop_phi1} and 
shown in Figure~\ref{fig:local}(b), ensures that $\psi$ is exactly the procedure to recover a 
plane bipolar orientation from its image $P=\phi(O)$. In other words:
\begin{lemma}\label{lem:psi1g}
For each plane bipolar orientation $O$, $\psi(\phi(O))=O$.
\end{lemma} 
Next we prove that $\psi$ is also the right inverse of $\phi$ when restricted to $N$-avoiding
structures.
\begin{lemma}\label{lem:psi1d}
For any plane bipolar poset $P$, the oriented map $O=\psi(P)$ is a plane bipolar
orientation. In addition $\phi(O)=P$ if $P$ is $N$-avoiding.
\end{lemma}
\begin{proof}
Let $v,v'$ be two vertices of $O$, and let $f,f'$ be the faces of $P$
corresponding respectively to $v$ and $v'$.
Clearly, if there is an oriented path
of edges of $O$ going from $v$ to $v'$, then $f$  is smaller than $f'$ for the left-to-right order on the faces
of $P$. Hence, $O$ is acyclic. In addition, each vertex $v\neq\{s,t\}$ of $O$ has indegree 
equal to the number of left lateral vertices and outdegree equal to the number of right lateral vertices of the corresponding inner face of $P$.
Hence, $s$ is the only source and $t$ the only sink of $O$, so $O$ is bipolar.

Now assume that $P$ is $N$-avoiding, 
and let us now show that $\phi(O)=P$. Recall that 
the first step of $\phi$ is to add two edges from $s$ to $t$,
yielding an augmented map $\tO$.
 Consider the oriented map $M$ obtained by superimposing $\tO$ and $P$, that is, 
 there is a 
vertex of $P$ in the middle of each edge of $\tO$ and a vertex of $\tO$ inside each face of $P$,
see Figure~\ref{fig:phi1}(b). 
Call $P$-vertices and $P$-edges the vertices and edges of $M$ that come from $P$,
and call $O$-vertices and $O$-edges the vertices and edges of $M$ that come
from $\tO$. Note that a $P$-edge corresponds to an edge of $P$, but an $O$-edge
corresponds to a half-edge of $\tO$.

Let $v$ be a $P$-vertex. 
By definition of $\psi$ (and by Property (V) of plane bipolar orientations), 
the edges of $M$ incident to $v$ form in clockwise order: 
a nonempty interval of outgoing $P$-edges, one ingoing
$O$-edge, a nonempty interval of ingoing $P$-edges, and one outgoing $O$-edge. 
Hence any $P$-edge  inside an inner face $f$ of $\tO$ must go from a $P$-vertex in the left
lateral path of $f$ to a $P$-vertex in the right lateral path of $f$.
Thus the $P$-edges  inside $f$ are naturally ordered from down to top; in this order, two 
consecutive edges are said to be \emph{meeting} if they share a vertex and \emph{parallel} otherwise.
By a simple counting argument, the number of $P$-edges 
inside $f$ is $\mathrm{deg}(f)-1-p(f)$,
where $p(f)$ is the number of parallel consecutive $P$-edges inside $f$ and $\mathrm{def}(f)$ is the degree of $f$. As in the proof of Lemma~\ref{lem:prop_phi1}, write $V(M)$, $E(M)$, and $F(M)$ for the sets of vertices, edges, and inner faces of a map $M$.
From the above discussion, the total number of $P$-edges satisfies
\begin{equation}\label{eq:P}
|E(P)|=\sum_{f\in F(\tO)}\mathrm{deg}(f)-1-p(f)=2|E(\tO)|-2-|F(\tO)|-\sum_{f\in F(\tO)}p(f).
\end{equation}
Moreover, $|E(P)|=|V(P)|+|F(P)|-1$ by the Euler relation, $|V(P)|=|E(\tO)|$ by definition
of $\phi$, and $|F(P)|=|V(\tO)|-2$ by definition of $\psi$. Hence $|E(P)|=|E(\tO)|+|V(\tO)|-3$, which is equal to $2|E(\tO)|-|F(\tO)|-2$ again by the Euler relation.
Hence, from the first expression~\eqref{eq:P} of $|E(P)|$ obtained above, we conclude that
  $\sum_{f\in F(\tO)}p(f)=0$, so all the 
$p(f)$ are zero, so that all consecutive $P$-edges in any inner face of $\tO$ are meeting.
As $P$ has no $N$-pattern and all consecutive edges are meeting, the only possible configuration
for the $P$-edges inside a face of $\tO$ is the configuration shown in Figure~\ref{fig:local}(a). In other 
words, the $P$-edges are those inserted when applying the mapping $\phi$
to $O$. Hence $\phi(O)=P$.
\end{proof}

From Lemma~\ref{lem:psi1g} and Lemma~\ref{lem:psi1d} we obtain our first bijection:
\begin{prop}\label{theo:phi1_bij}
For $n\geq 1$ and $i\geq 0$, 
the  mapping $\phi$ is a bijection between plane bipolar orientations with $n$ edges
and $i$ non-special vertices, and
$N$-avoiding plane bipolar posets with $n$ non-special vertices and $i$ inner faces. 
The inverse mapping of $\phi$ is $\psi$.
\end{prop}

\begin{remark}
As recalled in Section~\ref{sec:plane_bip}, any bipolar orientation $O$ on a graph $G=(V,E)$ (planar or not) 
gives rise to a partial order $P=\beta(O)$ on $E$, called edge-poset,
according to the precedence order
of the edges along oriented paths (observe that no embedding is needed to define this poset). 
It is well known that the mapping $\beta$ is indeed
a bijection between bipolar orientations on graphs and  
so called \emph{$N$-free posets}, which are posets with no induced $N$ in the Hasse diagram~\cite{HaMo87}, that is, no quadruple $(v_1,v_2,v_3,v_4)$ of vertices such that
the only adjacencies on these vertices are: an edge from $v_1$ to $v_3$, an edge
from $v_2$ to $v_4$ and an edge from $v_1$ to $v_4$. 
Moreover, it is possible to enrich this construction so as to take a planar embedding into account:
the planar embedding of the plane bipolar orientation turns to a so-called 2-realizer 
of the associated poset~\cite{FrOs96}.
   
Our mapping $\phi$ shares some resemblance
with this classical construction, 
but there are significant differences. First, the bipolar poset obtained 
from $\phi$   might 
have induced $N$ in the form of \rN-patterns (due to the embedding, an induced $N$
can appear either as an $N$-pattern or as an \rN-pattern).  
Second, our bijection has the nice feature that 
it operates directly on the embedding, in both directions ($\phi$ and $\psi$). 
The mapping $\phi$ is actually just a simple
adaptation of $\beta$ so as to take the embedding into account.
In contrast, $\phi^{-1}=\psi$ is not just an adaptation of $\beta^{-1}$. Indeed $\psi$
is significantly simpler, it operates on the embedding in a local way, 
whereas $\beta^{-1}$  requires some non-local 
 manipulations on the
edge-poset.
\end{remark}

\subsection{From $N$-avoiding plane bipolar posets to $N$-avoiding transversal structures.}\label{sec:phi2}
The next step is to complete an $N$-avoiding plane bipolar poset into a 
transversal structure. The procedure, called $\phi'$, is very similar to $\phi$. 
Precisely, given a plane bipolar poset $P$ ($N$-avoiding or not)
the associated transversal structure $X=\phi'(P)$ is defined as follows,  
see Figure~\ref{fig:phi1}(c)-(d):
\begin{itemize}
\item
\emph{Create the outer quadrangle.}
Insert one vertex, denoted $W$, in the left special face of $P$ and one vertex, denoted $E$,
in the right special face of $P$. Connect $W$ and $E$ to the source and sink of $P$, denoted respectively by $S$ and $N$, thus creating an outer quadrangle with vertices 
$(W,N,E,S)$ in clockwise order. 
The bipolar poset augmented from $P$ by insertion of the quadrangle ---edges of the quadrangle
oriented from $S$ toward $N$--- is denoted $\tP$. 
\item
\emph{Insert the blue edges.}
For each inner face $f$ of $\tP$, insert (in a planar way) blue edges inside $f$ 
so as to create the following adjacencies, see Figure~\ref{fig:local}(c):
left lateral vertices of $f$ are connected to the bottomright lateral vertex of $f$, and 
right lateral vertices of $f$ are connected to the topleft lateral vertex of $f$.
The inserted blue edges inside $f$ are directed from the left lateral vertices to the right lateral vertices of $f$.
In other words, $f$ is triangulated by transversal blue edges in the unique way avoiding
blue $N$-patterns inside~$f$.
\end{itemize}

\begin{lemma}\label{lem:psi2g}
Let $P$ be a plane bipolar poset. Then $\phi'(P)$ is a transversal 
structure whose red bipolar poset is $P$ and whose blue bipolar poset is $N$-avoiding.
\end{lemma}
\begin{proof}
Clearly $X:=\phi'(P)$ is a transversal structure (conditions (T1) and (T2) are satisfied) whose
red bipolar poset is $P$. The blue bipolar poset is $N$-avoiding due to the two 
following observations:
 i) two blue edges with the same origin are inside the same face of $P$, and two blue edges 
with the same end are inside the same face of $P$, hence any blue $N$-pattern must be  inside a face of $P$;
ii) there is no blue $N$-pattern inside a face of $P$. 
\end{proof}
Call $\psi'$ the mapping that associates to a transversal structure its red bipolar poset.

\begin{prop}\label{prop:phi2_bij}
For $n\geq 1$ and $i\geq 0$, 
the mapping $\phi'$ is a bijection between $N$-avoiding plane bipolar posets with $n$ 
non-special vertices and $i$ inner faces,
and $N$-avoiding transversal structures with $n$ inner vertices and $n+i+1$ red edges. 
The inverse mapping of $\phi'$ is
$\psi'$.
\end{prop}
\begin{proof}
According to Lemma~\ref{lem:psi2g}, $\psi'(\phi'(P))=P$
for any plane bipolar poset $P$, so $\psi'$ is the left inverse of $\phi'$. 
Let us show that $\psi'$ is also the right inverse of $\phi'$
when the mappings are restricted to $N$-avoiding structures. Let $X$ be an $N$-avoiding 
transversal structure, with $\Tr:=\psi'(X)$ its red bipolar poset and $\Tb$ its blue bipolar poset. 
When applying $\phi'$ to $\Tr$, each inner face $f$ of $\Tr$ is triangulated by transversal blue edges going from a left lateral vertex
to a right lateral vertex of $f$. It is easily checked that the only such $N$-avoiding configuration
is the one where all edges are incident either to the topleft or to the bottomright lateral vertex of $f$.
In other words, the blue edges inside $f$ are placed according to the insertion process of $\phi'$, hence $\phi'(P)=X$.  Finally, the parameter-correspondence is due to the Euler relation, which 
ensures that a plane bipolar poset with $n$ non-special vertices and $i$ inner faces has $n+i+1$
edges. 
\end{proof}

To conclude, we have described a bijection $\Phi=\phi'\circ\phi$ between
plane bipolar orientations with $n\geq 1$ edges and $i+2$ vertices
and $N$-avoiding transversal structures with $n$ inner vertices and $n+i+1$
red edges. The bijection $\Phi$ operates in two steps $\phi$ and $\phi'$,
the intermediate combinatorial structures being the $N$-avoiding plane bipolar
posets. 

\subsection{Counting Baxter families}\label{sec:baxter}
Plane bipolar orientations with $n$ edges and $i$ non-special vertices are known to be 
counted by the coefficients
\begin{equation}
\Theta_{n,i}:=\frac{2}{n(n+1)^2}\binom{n+1}{i}\binom{n+1}{i+1}\binom{n+1}{i+2}.
\end{equation}
The formula has first been proved by Rodney Baxter~\cite{baxter}, guessing the answer
from a recurrence satisfied by the counting sequence; 
a direct computation based on the 
``obstinate Kernel method''  
is due to Bousquet-M\'elou~\cite{bousquet-melou-four}. 
A bijective proof has been found recently~\cite{FuPoSc09}. 
Summing over $i$ the coefficients $\Theta_{n,i}$, the number of plane bipolar orientations with $n$ 
edges is 
$$\theta_n:=\sum_i\frac{2}{n(n+1)^2}\binom{n+1}{i}\binom{n+1}{i+1}\binom{n+1}{i+2}.$$
The numbers $\theta_n$, called the \emph{Baxter numbers} (after another Baxter, Glen Baxter, whose name is given to a family of permutations counted by $\theta_n$),  
surface recurrently in the enumeration of combinatorial structures, see~\cite{FeFuNoOr07} 
for a recent survey
on ``Baxter families''. 
Our bijections with $N$-avoiding structures, as described in Section~\ref{sec:bij_N_avoid}, bring new Baxter families to the surface:

\begin{prop}\label{prop:Theta}
For $n\geq 1$ and $i\geq 0$, the number $\Theta_{n,i}$ counts:
\begin{itemize}
\item
$N$-avoiding plane bipolar posets with $n$ non-special vertices and $i$ inner faces,
\item
$N$-avoiding transversal structures with $n$ inner vertices and $n+i+1$ red edges.
\end{itemize}
Taking into account the parameter $n$ only, these structures are counted by the Baxter number 
$\theta_n$.
\end{prop}

\section{Bijection between non-separable maps and irreducible triangulations}\label{sec:bij_2conn_irr}

\subsection{Non-separable maps as specific plane bipolar orientations}\label{sec:2conn_spec}
For rooted maps, plane bipolar orientations are always assumed to have the root edge going
from the source to the sink.
As already mentioned in Section~\ref{sec:plane_bip},
 bipolar orientations are naturally associated with the property
of non-separability; a rooted map is non-separable if and only if it admits a plane bipolar orientation. 
Even more is true, namely,  each rooted non-separable map can be endowed with a specific 
plane bipolar orientation in a canonical way. Given a plane bipolar orientation, we define a 
\emph{left-oriented piece}, shortly a LOP, as a 4-tuple $(v_1,v_2,f_1,f_2)$ made of two distinct 
vertices $v_1,v_2$
and two distinct faces $f_1,f_2$ of $O$ such that the following conditions are satisfied, 
see Figure~\ref{fig:LOP}(a):
\begin{itemize}
\item
the vertex $v_1$ is the sink of $f_2$ and is a left lateral vertex of $f_1$, 
\item
the vertex $v_2$ is a right lateral vertex of $f_2$ and is the source of $f_1$.
\end{itemize} 
Note that $v_1$ and $v_2$ must be non-special vertices and $f_1$ and $f_2$
must be inner faces.

\begin{fact}[\cite{DeOss,Oss}]\label{fact:2conn}
A rooted map is non-separable if and only if it admits a bipolar orientations. 
A rooted non-separable map $M$ has a unique plane bipolar orientation with no LOP, called 
the \emph{minimal plane bipolar orientation} of $M$.
\end{fact}
The terminology is due to the fact that the set of plane bipolar orientations of a fixed rooted non-separable map is a distributive lattice, the minimal element of the lattice being the bipolar orientation with no LOP,
see~\cite{Oss}. The set of plane bipolar orientations with no LOP is denoted
$\Omin$.


\begin{lemma}[link with plane bipolar posets]\label{lem:LOP_phi1}
Let $O$ be a plane bipolar orientation, and let $P=\phi(O)$ be the associated $N$-avoiding plane bipolar poset. Then $O$ has no LOP if and only if $P$ has no LOP.
\end{lemma}
\begin{proof}
Assume that $O$ has a LOP $(v_1,v_2,f_1,f_2)$. Let $w_1$ 
be the vertex of $P$ inserted  in the bottomright edge of $f_2$ and $w_2$ 
the vertex of $P$ in the topleft edge of $f_1$. Let $f_1'$ be the face of $P$
corresponding to $v_2$ and $f_2'$ the face of $P$ corresponding to $v_1$. 
 Then $(w_1,w_2,f_1',f_2')$ is easily checked to be a LOP
of $P$, as shown in Figure~\ref{fig:LOP}(b).  

Conversely, assume that $P$ has a LOP $(w_1,w_2,f_1'=f_{v_1},f_2'=f_{v_2})$, with $v_1$ and $v_2$ the vertices of $O$ associated respectively 
to the two faces $f_2'$ and $f_1'$ of $P$. 
Let $f_1$ be the face of $O$ containing the ingoing edges of $w_1$, and
let $f_2$ be the face of $O$ containing the outgoing edges 
of $w_2$. 
Let us first check that $f_1$ and $f_2$ can not be equal. Assume $f_1=f_2=f$,
then $w_1$ is in the right lateral path and $w_2$ is in the left lateral path of $f$;
since $P=\phi(O)$ and since there is a path from $w_2$ to $w_1$ (due to the
LOP of $P$), the vertices $w_2$ and $w_1$ must be adjacent, so the LOP of $P$
is in fact reduced to an $N$-pattern, a contradiction. Thus $f_1$ and $f_2$ are 
distinct. 
From the definition of $\psi$ and the fact that $w_2$ and $w_1$ are not
adjacent, it is easily checked that $(v_1,v_2,f_1,f_2)$ is a LOP of $O=\psi(P)$,
see Figure~\ref{fig:LOP}(b).   
\end{proof}

\begin{figure}
\begin{center}
\includegraphics[width=16cm]{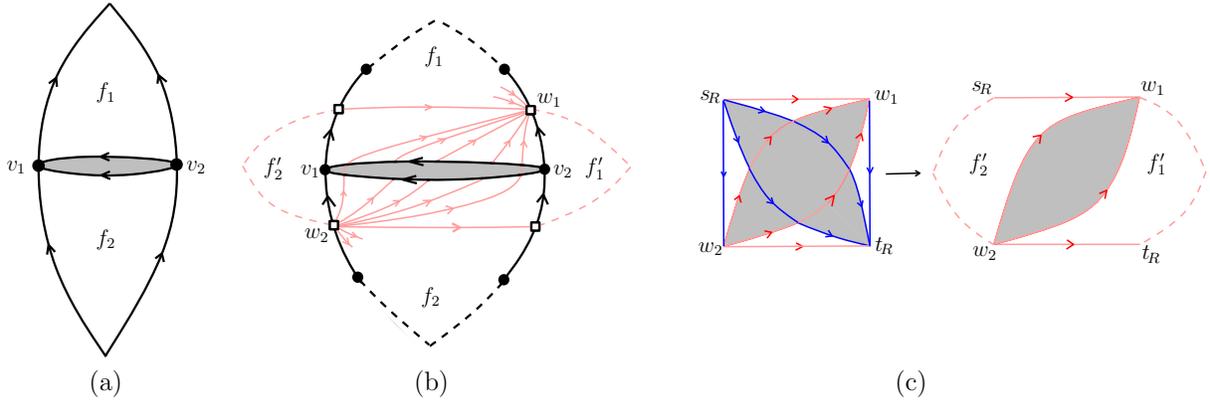}
\end{center}
\caption{(a) A left-oriented piece (LOP) of a plane bipolar orientation. (b) Each LOP in a plane bipolar orientation $O$ yields a LOP in $P=\phi(O)$.
(c) Each right alternating 4-cycle in an $N$-avoiding transversal structure $X$ yields a LOP in 
the red bipolar poset of $X$.}
\label{fig:LOP}
\end{figure}

Let $\Pmin$ be the set of plane bipolar posets with no LOP.
Note that an $N$-pattern---with $e$ its central edge---in 
a plane bipolar poset induces a LOP 
$(v_1,v_2,f_1,f_2)$
where $v_1$ is the end of $e$, $v_2$ is the origin of $e$, $f_1$ is the face on the right of $e$, and $f_2$ is the face on the left of $e$. Hence $\Pmin$ is also the set
of $N$-avoiding plane bipolar posets with no LOP, which by
Lemma~\ref{lem:LOP_phi1} is in bijection with $\Omin$. Thus we obtain 
\begin{equation}
\Omin\simeq_{\phi}\Pmin.
\end{equation}

\subsection{Irreducible triangulations as specific transversal structures}\label{sec:irr_tri_spec}
As we recall here from~\cite{Fu07b}, one can endow an 
irreducible triangulation with a specific transversal structure in a canonical way.
Given a transversal structure $X$, an \emph{alternating 4-cycle} of $X$ is a cycle of four edges that 
alternate in color, that is, the cycle has two opposite red edges and two opposite blue edges. 
Given $v$ a vertex of $R$, let $e$ and $e'$ be the two edges of $R$ incident to $v$,
such that $e'$ follows $e$ in a counterclockwise walk around $R$; $e$ is called the \emph{left-edge}  of $v$
and $e$ is called the \emph{right-edge} of $v$. A boundary incidence of $R$ is the incidence
of an edge $e$ inside $R$ with a vertex $v$ on $R$; it is called a \emph{left incidence} if 
$e$ has the color of the left-edge of $v$ and a \emph{right incidence} if $e$ has the color
of the right-edge of $v$. 

As proved in~\cite{Fu07b}, $R$ is of two possible types: either all boundary incidences for $R$ incidences
are left incidences, in which case $R$ is called a \emph{left alternating 4-cycle}, or all boundary incidences of $R$ are right incidences, in which case $R$ is called a \emph{right alternating 4-cycle}.  

Let us give a slightly different formulation. By Condition (T1), the cycle $R$
is made of two oriented paths of length 2 with the same origin $s_R$ and the same
end $t_R$, which are called the \emph{source} and the \emph{sink} of $R$.
Let $X_R$ be $X$ restricted to $R$ and its interior, forgetting the colors and
directions of the edges of $R$. Then, if $R$ has at least one interior vertex, $X_R$
is a transversal structure. If $s_R$ is source of red edges, $R$ is a left
alternating 4-cycle; if $R$ is source of blue edges, $R$ is a right alternating 4-cycle.
In the degenerated case where $R$ has no vertex inside---thus it has a unique edge $e$ inside--- consider the triple formed by $e$ and the two edges of $R$ of the same color as $e$. Then either the triple forms an $\rN$-pattern, in which case $R$
is a left alternating 4-cycle, or the triple forms an $N$-pattern, in which case $R$
is a right alternating 4-cycle.

\begin{fact}[\cite{Fu07b}]\label{fact:irr_minimal}
A triangulation of the 4-gon is irreducible if and only if it admits a transversal structure. An irreducible triangulation admits a unique transversal structure with no right alternating 
4-cycle, called its \emph{minimal} transversal structure.
\end{fact}

Again the terminology refers to the fact that the set of transversal structures of an irreducible 
triangulation is a distributive lattice whose minimal element is the transversal structure with
no right alternating 4-cycle, see~\cite{Fu07b}. 

\vspace{0.2cm}


\begin{lemma}[link with plane bipolar posets]\label{lem:LOP_phi2}
Let $X$ be an $N$-avoiding transversal structure, and let $\Tr$ be the red bipolar poset
 of $X$. Then $X$ has no right alternating 4-cycle if and only if $\Tr$ has no LOP.
\end{lemma}
\begin{proof}
Assume that $X$ has a right alternating 4-cycle $R$, the vertices of $R$ in clockwise order
being denoted $(s_R,w_1,t_R,w_2)$, with $s_R$ the source and $t_R$ the sink of $R$.
Let $f_1'$  
be the face of $\Tr$ containing the blue edge $(w_1,t_R)$, and let $f_2'$
be the face of $\Tr$ containing the blue edge $(s_R,w_2)$.  
As discussed at the end of the definition of right alternating 4-cycles, a right alternating
4-cycle with no vertex inside yields an $N$-pattern.
Hence, since $X$ is $N$-avoiding, $R$ has at least
one vertex inside, so the faces $f_1'$ and $f_2'$ are distinct.
It is easily checked that $(w_1,w_2,f_1',f_2')$ is a LOP, as shown in Figure~\ref{fig:LOP}(c).
Conversely assume that $\Tr$ has a LOP $(w_1,w_2,f_1',f_2')$. Let $s_R$ be the topleft lateral vertex of $f_2'$ , and let $t_R$ be the  bottomright lateral vertex of $f_1'$. Since $X$ is $N$-avoiding, $X$ is 
obtained from its red bipolar poset as $X=\phi'(\Tr)$. By definition of $\phi'$ (see Figure~\ref{fig:local}(c)),
in $X$ there is a blue edge from $s_R$ to $w_2$ and a blue edge from $w_1$
to $t_R$.  
Hence, the 4-cycle $R=(s_R,w_1,t_R,w_2)$ is an alternating 4-cycle. As there are 
red edges inside $R$ incident to $w_1$ and $w_2$, $s_R$ is source of
blue edges, so $R$ is a right alternating 4-cycle.  
\end{proof}

Let $\Tmin$ be the set of transversal structures with no right alternating 4-cycle
(minimal transversal structures). Note that an $N$-pattern in a transversal structure
yields a right alternating 4-cycle: the cycle delimited by the faces
on both sides of the central edge of the $N$-pattern. Hence $\Tmin$ is also 
the set of $N$-avoiding transversal structures with no right alternating 4-cycle, which
by Lemma~\ref{lem:LOP_phi2} is in bijection with $\Pmin$. Thus we obtain:
\begin{equation}
\Pmin\simeq_{\phi'}\Tmin.
\end{equation}

\subsection{The bijection}\label{sec:the_bij_F1}
To conclude, we have
$$
\Omin\simeq_{\phi}\Pmin\simeq_{\phi'}\Tmin.
$$
Hence $\Phi=\phi'\circ\phi$ specializes into a bijection  
between minimal 
plane bipolar orientations and minimal transversal structures. 
Since minimal plane bipolar orientations identify to (rooted) non-separable maps and
minimal transversal structures identify to (rooted ) irreducible triangulations, we obtain
a bijection, called $F_1$, between these two map families. Precisely, given a rooted non-separable map $M$ with at least 2 edges
$T:=F_1(M)$ is obtained as follows:
\begin{itemize}
\item
endow $M$ with its minimal plane bipolar orientation $O$, the root edge being deleted,
\item
compute the transversal structure $X$ associated to $O$:  $X:=\Phi(O)$,
\item
return $T$ as the irreducible triangulation underlying $X$, rooted at the edge going from 
$N$ to $W$,
\end{itemize}
and $M=F_1^{-1}(T)$ is obtained
as follows:
\begin{itemize}
\item
endow $T$ with its minimal transversal structure,
\item
compute the plane bipolar orientation $O$ associated to $X$: $O:=\psi(\psi'(X))$,
\item
return the rooted non-separable map $M$ underlying $O$ ($M$ receives an additional
root edge going from the source to the sink).
\end{itemize} 

\begin{thm}\label{theo:F1}
For $n\geq 2$, the mapping $F_1$ is a bijection between rooted non-separable maps with $n$ edges and
irreducible triangulations with $n+3$ vertices. 
\end{thm}
\begin{proof}
The bijection results from $\Omin\simeq\Tmin$, as discussed above.  
The parameter-correspondence is inherited from the parameter-correspondences of 
$\phi$ and $\phi'$.
\end{proof}


\subsection{Counting non-separable maps}\label{sec:map_families}
Brown and Tutte~\cite{BT64} used the recursive approach introduced by Tutte to show  that the number of rooted non-separable maps
with $n+1$ edges and $i+2$ vertices is
\begin{equation}
\Lambda_{n,i}=\frac{(n+i)!(2n-i-1)!}{(i+1)!(n-i)!(2i+1)!(2n-2i-1)!}.
\end{equation} 
And the number of rooted non-separable with $n+1$ edges is $\lambda_n=\sum_{i=0}^{n-1}\Lambda_{n,i}$, which simplifies to
\begin{equation}
\lambda_n=\frac{2(3n)!}{(n+1)!(2n+1)!}.
\end{equation}
Bijective proofs have been given later on, the first one being a correspondence
with ternary trees bearing labels so as to satisfy a positivity condition~\cite{JaSc98}.
A more direct construction has been described by Schaeffer in his PhD~\cite{S-these}, as a 
 4-to-$(2n+2)$ correspondence between ternary trees with $n$ nodes and 
rooted non-separable maps with $n+1$ edges. The correspondence is based on local ``closure''
operations on a ternary tree, with the effect of closing one by one the faces of the associated map.

Our bijections relating maps via specific combinatorial structures, as described in Section~\ref{sec:bij_2conn_irr}, yields the following counting results:

\begin{prop}
For $n\geq 1$ and $i\geq 0$, the number $\Lambda_{n,i}$ counts:
\begin{itemize}
\item
rooted non-separable maps with $n+2$ vertices, $i+2$ faces, and whose  minimal plane bipolar orientation (forgetting the root-edge) is a plane bipolar poset,
\item
rooted irreducible triangulations with $n$ inner vertices and whose minimal transversal structure has 
$n+i+1$ red edges. 
\end{itemize}
Taking into account the parameter $n$ only, these map families are counted by $\lambda_n$.
\end{prop}
\begin{remark}
As for rooted non-separable maps, there exists a simple 4-to-$(2n+2)$ correspondence 
between ternary trees with $n$ nodes and rooted irreducible triangulations with $n$ 
inner vertices~\cite{Fu07b}, again based on ``closure'' operations on the tree.

It turns out that our bijection $F_1$ puts in correspondence a non-separable map $M$ and an irreducible
triangulation $T$ that have the same underlying (unrooted) ternary tree. This is no surprise, as 
the 
ternary tree associated to $M$ ($T$) is a spanning tree that arises from deleting specific edges in the 
minimal bipolar orientation of $M$ (minimal transversal structure of $T$, resp.), whereas $F_1$ 
results from $\phi$ matching minimal plane bipolar orientations with minimal
transversal structures. 
Hence, $F_1$ can be considered as a direct bijective construction on maps that has the same effect
as computing the ternary tree associated to a non-separable map and then closing the tree into
an irreducible triangulation.
\end{remark}

\section{Bijection between loopless maps and triangulations}\label{sec:bij_loopless_triang}

\subsection{Decomposing loopless maps into non-separable components}\label{sec:dec_2conn}
It is well known in graph theory that a connected graph $G$ is decomposed into a  collection of non-separable components called the \emph{blocks}, which are the maximal non-separable subgraphs of $G$~\cite{Ha,Mohar}. Conversely, the process of 
gluing non-separable graphs at common vertices in a tree-like fashion yields any connected graph in a unique way. 
This classical decomposition readily adapts 
to rooted loopless maps~\cite{Tu63,GoJa83}. For our purpose
it proves convenient to define the size $|M|$ of a rooted loopless map as its number of edges including the root. In this section 
we consider the vertex-map---made of a unique vertex and no edge---as a 
rooted loopless map. In contrast, non-separable maps are required to have at least
one edge. 

\begin{fact}\label{fact:dec_2conn}
The following process:
\begin{itemize}
\item
take a rooted non-separable map $C$, and order the corners of $C$ in a canonical way as $\alpha_1,\ldots,\alpha_{2|C|}$; 
\item
take a collection $M_1,\ldots,M_{2|C|}$ of rooted loopless maps; for each $1\leq i\leq 2|C|$, insert
$M_i$ in the corner $\alpha_i$ of $C$, merging the outer face of $M_i$ with the face incident to $\alpha_i$ and the root vertex of $M_i$ with the vertex incident to 
$\alpha_i$,
\end{itemize}
yields a rooted loopless map, and 
each rooted loopless map $M$ is obtained in a unique way from this process.

The map $C$ is called the \emph{core-map} of $M$, 
and the ($2|C|+1$)-tuple $\frak{D}(M)=(C;M_1,\ldots,M_{2|C|})$ is called the
\emph{decomposition-tuple} of $M$. (Note that the size of $M$ is the sum of the sizes of the maps in the decomposition-tuple.)
\end{fact}

\subsection{Decomposing triangulations into irreducible components}\label{sec:dec_irr}
Similarly as for loopless maps, there exists a classical decomposition of triangulations into 
components of higher connectivity: each triangulation on the topological sphere is
obtained from a collection of 4-connected triangulations (which are triangulations where all
3-cycles are facial) glued at common triangles in a tree-like fashion. The same idea ---decomposition
at separating 3-cycles--- works as well to decompose rooted triangulations of the 4-gon
into rooted irreducible components~\cite{T62a}. First let us introduce some terminology. Given a rooted
triangulation $T$ of the 4-gon, denote by $N, E, S, W$ the outer vertices of $T$ in clockwise order
around the outer face, with $N$ the root-vertex; 
then $T$ is called \emph{$W\!E$-diagonal} 
if $W$ and $E$ are adjacent, $T$ 
is called \emph{$SN$-diagonal} if $S$ and $N$ are adjacent, and $T$ is called \emph{non-diagonal} otherwise.
Notice that the two diagonal cases are disjoint by planarity of $T$.
The unique rooted irreducible triangulation that is $W\!E$-diagonal  is the quadrangle
$(N,E,S,W)$ split by a diagonal $(W,E)$; this map is 
called the \emph{$W\!E$-link-map}. 
Similarly the unique rooted irreducible triangulation that is $SN$-diagonal  is the quadrangle
$(N,E,S,W)$ split by a diagonal $(S,N)$; this map is called the \emph{$SN$-link-map}.
For our purpose it is
convenient to define the size of a triangulation or of a triangulation of the 4-gon $T$ as $||T||:=(|V(T)|-3)$.
Notice that $||T||$ is the number of inner vertices if $T$ is a triangulation and $2||T||$ is the number
of inner faces if $T$ is a triangulation of the 4-gon, by the Euler relation.

\begin{fact}\label{fact:dec_irr}
The following process:
\begin{itemize}
\item
take a rooted irreducible triangulation $I$, and order the inner faces of $I$ in a canonical way
as $f_1,\ldots,f_{2||I||}$;
\item
take a collection $T_1,\ldots,T_{2||I||}$ of rooted triangulations; 
for each $1\leq i\leq 2||I||$, 
substitute $f_i$ by $T_i$ in a canonical way (e.g., the outer triangle of $T_i$ fits with the contour
of $f_i$ and the root vertex of $T_i$ fits with a distinguished vertex of $f_i$)
\end{itemize}
yields a rooted triangulation of the 4-gon; and
each rooted triangulation $T$ of the 4-gon is obtained in a unique way from this process.

The map $I$ is called the \emph{core-triangulation} of $T$; $T$ is non-diagonal if $I$ has at least
one inner vertex, is $W\!E$-diagonal if $I$ is the $W\!E$-link-map, and is $SN$-diagonal if $I$
is the $SN$-link-map.
The $(2||I||+1)$-tuple $\frak{D}(T)=(I;T_1,\ldots,T_{2||I||})$ is called the \emph{decomposition-tuple} of $T$.
(Note that the size of $T$ is the sum of the sizes of the maps in the decomposition-tuple.) 
\end{fact}

As a corollary we can describe the decomposition-tuples corresponding to 
rooted triangulations. 
Indeed, in each size $k\geq 1$, there is a simple bijection between 
rooted triangulations  and rooted triangulations of the 4-gon
that are not $WE$-diagonal. Starting from a rooted triangulation $T$, delete the 
edge following the root in counterclockwise order around the outer face; the obtained
map is a triangulation of the 4-gon, and $W$ and $E$ are not adjacent, since the 
unique edge connecting them in $T$ has been deleted. Conversely, starting from 
a rooted triangulation $\tT$ of the 4-gon that is not $WE$-diagonal, 
add an edge between
$E$ and $W$. Since $W$ and $E$ are not adjacent in $\tT$, the obtained map $T$ 
has no multiple edges (nor loops) and has clearly all faces of degree 3, so $T$
is a triangulation.

\begin{cor}\label{coro:triang} 
In each size $k\geq 1$, rooted triangulations  can be identified with 
rooted triangulations of the 4-gon that are not $WE$-diagonal, which
 themselves identify to decomposition-tuples 
whose core-triangulation is not the $WE$-link-map.
\end{cor}

\subsection{The bijection}\label{sec:the_bij_F2}
Our size-preserving bijection, called $F_2$, between rooted loopless maps and rooted triangulations
is defined recursively from the bijection $F_1$.
Recall that $F_1$ maps a non-separable map of size at least 2 (at least one non-root
edge) to an irreducible triangulation of the same size (at least one inner vertex).
 It is convenient here to augment the correspondence $F_1$ 
 into a bijection $\tFone$ 
 with one object of size $1$ on each side: the edge-map---made of two vertices connected by an edge---corresponds to the $SN$-link-map. 
 We denote by $\cL$ the family of (rooted) loopless maps including the vertex-map, 
 by $\cL_{\leq n}$
 the set of loopless maps of size at most $n$, and by $\cL_n$ the set
 of loopless maps of size $n$. Similarly, denote by $\cT$ the family of (rooted) triangulations including the triangle, by $\cT_{\leq n}$
 the set of triangulations of size at most $n$, and by $\cT_n$ the set
 of triangulations of size $n$. Finally, denote by $\cN$ the family of (rooted)
 non-separable maps with at least one edge and by $\cI$ the family of
 (rooted) irreducible triangulations with at least one inner vertex. Note that
 $\tFone$ is a size-preserving bijection between $\cN$ and $\cI\cup\{SN$-link-map$\}$. 
 
The size-preserving bijection $F_2$ between $\cL$ and $\cT$ 
is specified recursively as follows. 
First, the vertex-map---the unique rooted loopless map
$M$ such that $|M|=0$---is mapped by $F_2$ to the triangle-map---the unique rooted triangulation
$T$ such that $||T||=0$. Given $n\geq 0$, assume that $F_2$
 is a well defined bijection from $\cL_{\leq n}$ to $\cT_{\leq n}$, 
 that is size-preserving: $||F_2(M)||=|M|$ for every  $M\in\cL_{\leq n}$. 
 
 Let us now extend the bijection $F_2$ to maps of size $n+1$.
 By Fact~\ref{fact:dec_2conn}, 
 a tuple $(C;M_1,\ldots,M_{2|C|})$ corresponds to a loopless
 map in $\cL_{n+1}$ if and only if $C$ is in $\cN$, the $M_i$'s are in $\cL$,
 and $|C|+\sum_i|M_i|=n+1$, in which case all the $M_i$'s must be in $\cL_{\leq n}$.
 Let $\cDL_{n+1}$ be the set of such decomposition-tuples, note that 
 $\cDL_{n+1}\simeq\cL_{n+1}$.
 Similarly, by Fact~\ref{fact:dec_irr} and Corollary~\ref{coro:triang}, 
 a tuple $(I;T_1,\ldots,T_{2||I||})$ corresponds to a triangulation  in $\cT_{n+1}$ if and only if $I$ is in $\cI\cup\{SN\mathrm{-link-map}\}$, the $T_i$'s are in $\cT$,
 and $||I||+\sum_i||T_i||=n+1$, in which case all the $T_i$'s must be in $\cT_{\leq n}$.
 Let $\cDT_{n+1}$ be the set of such decomposition-tuples, note that 
 $\cDT_{n+1}\simeq\cT_{n+1}$.
Since $\tFone$ is a size-preserving bijection between $\cN$ and $\cI\cup\{SN$-link-map$\}$ and since, by induction, $F_2$ is a bijection between $\cL_{\leq n}$
and $\cT_{\leq n}$, we obtain a bijection $\cDL_{n+1}\simeq\cDT_{n+1}$:
for $(C;M_1,\ldots,M_{2|C|})\in\cDL_{n+1}$, the associated tuple in $\cDT_{n+1}$
is $(\tFone(C);F_2(M_1),\ldots,F_2(M_{2|C|}))$. Since $\cDL_{n+1}\simeq\cL_{n+1}$
and $\cDT_{n+1}\simeq\cT_{n+1}$, we conclude that $\cL_{n+1}$ is in bijection
with $\cT_{n+1}$, and we call $F_2$ the bijection extended to $\cL_{\leq n+1}$.

 \begin{thm}\label{theo:F2}
 For $n\geq 0$, the mapping $F_2$ is a bijection between rooted loopless maps with $n$ edges
 and rooted triangulations with $n$ inner vertices. 
  \end{thm}
 \begin{remark}
 Let us make a few comments on this bijection. Denote by $\cL_{n,k}$ the family of rooted loopless maps with $n$ edges and root-vertex of degree $k$, and
 denote by $\cT_{n,k}$ the family of rooted triangulations with $n$ inner vertices and 
 root-vertex of degree $k+2$. 
 The equality $|\cL_n|=|\cT_n|$ is known for long~\cite{T62a,Tu73},
 as both $|\cL_n|$ and $|\cT_n|$ are equal to 
 \begin{equation}
 a_n:=\frac{2(4n+1)!}{(n+1)!(3n+2)!}.
 \end{equation}
 A first bijective proof of $|\cL_n|=|\cT_n|$, more precisely of $|\cL_{n,k}|=|\cT_{n,k}|$, has been
 found by Wormald~\cite{Wo80}, 
 based on the observation that the generating trees (induced by root-edge-deletion) 
 of $\cL=\cup_{n,k}\cL_{n,k}$
 and of $\cT=\cup_{n,k}\cT_{n,k}$ are isomorphic. The bijection is thus recursive; a rooted loopless
 map is first completely decomposed to find out the place it occupies in the generating
 tree of $\cL$, then the associated rooted triangulation is the one occupying the same place
 in the generating tree of $\cT$.
 
 Our bijection has the original feature of being a mixing of a recursive construction (the parallel recursive decompositions into non-separable/irreducible components) and
  a direct construction (the 
 mapping $F_2$ to match non-separable/irreducible components). 
 The advantage is that, in practice,
 there should be fewer levels of recursion in our bijection than in Wormald's one. Indeed, as shown
 in~\cite{BaFlScSo01}, 
 a typical loopless map is made of a giant non-separable component ---of linear size---
 having small connected components attached in each corner.
 
 Let us also mention that direct bijective proofs of $|\cT_n|=a_n$ have been found recently~\cite{PS03b,FuPoScL},
 again based on closure operations on trees. However, no bijection with trees is 
 known yet to explain $|\cL_n|=a_n$. It would be interesting to find out such a construction
 for loopless maps, and hopefully derive from it a direct (nonrecursive) bijection between
 $\cL_n$ and $\cT_n$. 
 \end{remark}

\vspace{0.2cm}

\noindent\emph{Acknowledgements.} I am very grateful 
to Gilles Schaeffer for his encouragements to find
a new bijection between loopless maps and triangulations. 
I thank the anonymous referee for a very detailed report that
has led to a significant improvement of the presentation of the results. 
This work has also benefited from  interesting discussions with 
Mireille Bousquet-M\'elou, Nicolas Bonichon, and Nick Wormald.

\bibliography{mabiblio.bib}
\bibliographystyle{alpha}
\end{document}